\documentclass[12pt]{amsart}
\usepackage{verbatim}
\usepackage{graphicx}
\usepackage{xypic}
\usepackage{amssymb}
\usepackage{amsfonts}
\usepackage{amsmath}
\usepackage{hyperref}

\def\R{\mathbb{R}}
\def\RR{\mathbb{R}}

\def\N{\mathbb{N}}

\newtheorem{theorem}{Theorem}[section]
\newtheorem{lemma}[theorem]{Lemma}
\newtheorem{thm}[theorem]{Theorem}
\newtheorem{prop}[theorem]{Proposition}
\newtheorem{cor}[theorem]{Corollary}
\theoremstyle{definition}

\newtheorem{ex}[theorem]{Example}
 
 \DeclareMathOperator{\cc}{cc}

\numberwithin{equation}{section}
\newcommand{\bn}{\par\bigskip\noindent}
\newcommand{\pars}{\par\smallskip}
\newcommand{\adresse}{\par\bigskip \small\rm
Cimpri\v c, Jaka\par
Faculty of Mathematics and Physics\par
University of Ljubljana\par
Jadranska 21, SI-1000 Ljubljana, Slovenija\par
email: cimpric@fmf.uni-lj.si
\pars \pars \pars
Marshall, Murray\par
Department of Mathematics and Statistics\par
University of Saskatchewan\par
Saskatoon, SK S7N 5E6, Canada\par
email: marshall@math.usask.ca
\pars \pars \pars
Netzer, Tim\par
Fachbereich Mathematik und Statistik\par
Universit\"{a}t Konstanz\par
D-78457 Konstanz, Germany\par
email: Tim.Netzer@uni-konstanz.de}

\begin{document}
\title{Closures of quadratic modules}
\author{Jaka Cimpri\v c, Murray Marshall, Tim Netzer}

\date{April 4, 2009}

\begin{abstract}
We consider the problem of determining the closure $\overline{M}$
of a quadratic module $M$ in a  commutative $\Bbb{R}$-algebra with
respect to the finest locally convex topology. This is of interest
in deciding when the moment problem is solvable \cite{S1}
\cite{S2} and in analyzing algorithms for polynomial optimization
involving semidefinite programming \cite{L}. The closure of a
semiordering is also considered, and it is shown that the space
$\mathcal{Y}_M$ consisting of all semiorderings lying over $M$
plays an important role in understanding the closure of $M$. The
result of Schm\"udgen for preorderings in \cite{S2} is
strengthened and extended to quadratic modules. The extended
result is used to construct an example of a non-archimedean
quadratic module describing a compact semialgebraic set that has
the strong moment property. The same result is used to obtain a
recursive description of $\overline{M}$ which is valid in many
cases.
\end{abstract}

\maketitle

In Section 1 we consider the general relationship between the
closure  $\overline{C}$ and the sequential closure $C^{\ddagger}$
of a subset $C$ of a real vector space $V$ in the finest locally
convex topology. We are mainly interested in the case where $C$ is
a cone in $V$. We consider cones with non-empty interior and cones satisfying $C\cup -C = V$.

In Section 2 we begin our investigation of the closure
$\overline{M}$ of a quadratic module $M$ of a commutative
$\Bbb{R}$-algebra $A$; the focus is on  finitely generated
quadratic modules in finitely generated algebras. The closure of a
semiordering $Q$ of $A$ is also considered, and it is shown that
the space $\mathcal{Y}_M$ consisting of all semiorderings of $A$
lying over $M$ plays an important role in understanding the
closure of $M$; see Propositions \ref{2.1}, \ref{2.2} and \ref{2.3}. The
result of Schm\"udgen for preorderings in \cite{S2} is strengthened and extended to
quadratic modules; see Theorem \ref{new}.

In Section 3 we consider the case of quadratic modules that
describe compact semialgebraic sets. We use Theorem \ref{new} to
deduce various results; see Theorems \ref{3.1} and \ref{c}; and also to construct an example where
$\mathcal{K}_M$ is compact, $M$ satisfies the strong moment
property (SMP), but $M$ is not archimedean; see Example \ref{3.4}.

Theorem \ref{new} is also used in Section 4, to obtain a recursive
description of $\overline{M}$ which although it is not valid in general; see Example \ref{couex}; is valid in many cases; see
Theorem \ref{4.5}.

In Section 5, which is an appendix to Section 1, we give an
example of a cone $C$ where the increasing sequence of iterated
sequential closures $$C \subseteq C^{\ddagger} \subseteq
(C^{\ddagger})^{\ddagger} \subseteq \cdots$$ terminates after
precisely $n$ steps. In the case of quadratic modules and
preorderings, nothing much is known about the sequence of iterated
sequential closures beyond the example with $M^{\ddagger} \ne
\overline{M}$ given in \cite{N3}.

\section{Closures of Cones}

Consider a real vector space $V$. 
A convex set $U\subseteq V$ is
called \textit{absorbent}, if for every $x\in V$ there exists
$\lambda> 0$ such that $x\in\lambda U.$ $U$ is called
\textit{symmetric}, if $\lambda U\subseteq U$ for all
$|\lambda|\leq 1.$
The set of all convex, absorbent and symmetric subsets of $V$ forms
a zero neighborhood base of a vector space topology on $V$ (see
\cite[II.25]{d} or \cite{S}). This topology is called the \textit{finest locally convex
topology} on $V$.  
$V$ endowed with this topology is hausdorff, each linear functional on $V$ is
continuous, and each finite dimensional subspace of $V$ inherits
the euclidean topology.  

Let $C$ be a subset of $V$ and denote by
$C^{\ddagger}$ the set of all elements of $V$ which are
expressible as the limit of some sequence of elements of $C$. By
\cite[Ch. 2, Example 7(b)]{S}, every converging sequence in $V$
lies in a finite dimensional subspace of $V$, so $C^{\ddagger}$ is
just the union of the $\overline{C\cap W}$, $W$ running through
the set of all finite dimensional subspaces of $V$. (Observe: Each
such $W$ is closed in $V$, so $\overline{C\cap W}$ is just the
closure of $C\cap W$ in $W$.) We refer to $C^{\ddagger}$ as the
\it sequential closure \rm of $C$. Clearly $C \subseteq
C^{\ddagger} \subseteq \overline{C}$, where $\overline{C}$ denotes
the closure of $C$. For any subset $C$ of $V$ we have a
transfinite increasing sequence of subsets
$(C_{\lambda})_{\lambda\ge 0}$ of $V$ defined by $C_0 = C$,
$C_{\lambda^+} =(C_{\lambda})^{\ddagger}$, and $C_{\mu} =
\cup_{\lambda <\mu} C_{\lambda}$ if $\mu$ is a limit ordinal.
Question: Can one say anything at all about when this sequence
terminates?  We return to this point later; see the appendix at
the end of the paper.

We are in particular interested in the case where the dimension of
$V$ is countable. In this case, a subset $C$ of $V$ is closed if
and only if $C\cap W$ is closed in $W$ for each finite dimensional
subspace $W$ of $V$ \cite[Proposition 1]{B}. So $C^{\ddagger} =C$
if and only if $C$ is closed. Thus the sequence of iterated
sequential closures of $C$ terminates precisely at $\overline{C}$.

For the time being, we drop the assumption that $V$ is of
countable dimension. We are in particular interested in the case
when $C$ is a cone of $V$, i.e. if $C+C\subseteq C$ and
$\Bbb{R}^+\cdot C\subseteq C$ holds. In this case $C^{\ddagger}$
and $\overline{C}$ are also cones. Every cone is a convex set.
If $U$ is any convex open set
in $V$ such that $U\cap C = \emptyset$ then, by the Separation
Theorem \cite[II.39, Corollary 5]{d} (or \cite[Theorem 3.6.3]{M} in
the case of countable dimension), there exists a linear map $L : V
\rightarrow \Bbb{R}$ such that $L\ge 0$ on $C$ and $L<0$ on $U$.
This implies $\overline{C} = C^{\vee\vee}$. Here, $C^{\vee} $ is
the set of all linear functionals $L : V \rightarrow \Bbb{R}$ such
that $L(v)\ge 0$ for all $v\in C$ and $C^{\vee\vee} $ is the set
of all $v\in V$ such that $L(v)\ge 0$ for all $L \in C^{\vee}$.

\begin{prop}\label{1.1}
Let $C$ be a cone in $V$ and let $v \in V$. The following
are equivalent:
\begin{enumerate}
\item $v$ is the limit of a sequence of elements of $C$.
\item $\exists$ $q\in V$ such that $v+\epsilon q \in C$ for each real $\epsilon >0$.
\end{enumerate}
\end{prop}

\begin{proof} (2) $\Rightarrow$ (1). Let $v_i = v+\frac{1}{i}q$, $i=1,2,\cdots$.
Then $v_i\in C$ and $v_i\rightarrow v$ as $i\rightarrow \infty$. (1) $\Rightarrow$ (2).
Let $v = \lim_{i\rightarrow \infty} v_i$, $v_i\in C$. As explained earlier, the subspace of $V$
spanned by $v_1,v_2,\cdots$ is finite dimensional. Let $w_1,\dots,w_N \in C$ be a basis for this subspace.
Then $v_i = \sum_{j=1}^N r_{ij}w_j$, $v= \sum_{j=1}^N r_jw_j$, $r_{ij}, r_i \in \Bbb{R}$, $r_j = \lim_{i\rightarrow \infty} r_{ij}$.
Let $q:= \sum_{j=1}^N w_j$. Then, for any real $\epsilon >0$, $r_{ij}<r_j+\epsilon$ for $i$ sufficiently large,
so $v+\epsilon q = \sum_{j=1}^N (r_j+\epsilon)w_j = \sum_{j=1}^N r_{ij}w_j+\sum_{j=1}^N (r_j+\epsilon-r_{ij})w_j = v_i+\sum_{j=1}^N (r_j+\epsilon-r_{ij})w_j \in C$.
\end{proof}

\begin{cor}\label{1.2} If $C$ is a cone of $V$ then $$C^{\ddagger} = \{ v\in V \mid \exists q \in V  \text{ such that } v+\epsilon q
\in C \text{ for all real } \epsilon >0\}.$$
\end{cor}

The proof of Proposition \ref{1.1} shows we can always choose
$q\in C$.  In fact, we can find a finite dimensional subspace $W$
of $V$ (namely, the subspace of $V$ spanned by $w_1,\dots,w_N$)
such that $q\in W$ and $q$ is an interior point of $C\cap W$.

Cones with non-empty interior are of special interest. For a subset $C$ of $V$, a vector $v \in C$ is called an \it algebraic interior point of $C$ \rm if  for all $w \in V$ there is
a real $\epsilon >0$ such that $v+\epsilon w \in C$. 

\begin{prop}\label{x} \
\begin{enumerate}
\item  Let $C$ be a convex set in $V$. A vector $v \in C$ is an interior point of $C$ iff $v$ is an algebraic interior point of $C$.
\item Let $q$ be an interior point of a cone $C$ of $V$. 
If $v \in \overline{C}$ then $v+\epsilon q$ is an interior point of $C$ for all real $\epsilon >0$. 
\item If $C$ is a cone of $V$ with non-empty interior,
then $C^{\ddagger} = \overline{C}=\overline{\operatorname{int}(C)}=\operatorname{int}(C)^{\ddagger}$.
\end{enumerate}
\end{prop}

\begin{proof} (1) Let $v \in C$ be an algebraic interior point. Translating, we can assume $v=0$. Fix a basis $v_i$, $i\in I$ for $V$ and real $\epsilon_i >0$
such that $\epsilon_iv_i$ and $-\epsilon_iv_i$ belong to $C$. Take $U$ to be the convex hull of the set
$\{ \epsilon_iv_i, -\epsilon_iv_i \mid i \in I\}$. $U$ is convex, absorbent and symmetric and $0 \in U \subseteq C$.
The converse is clear.

(2) If $q \in \operatorname{int}(C)$ and $v \in \overline{C}$ then
$\lambda v+(1-\lambda)q \in \operatorname{int}(C)$  for all $0\le \lambda <1$, by  \cite[chapter III, Lemma 2.4]{BV} or \cite[page 38, 2.1.1]{S}.
Applying this with $\lambda = \frac{1}{1+\epsilon}$ and multiplying by $1+\epsilon$ yields $v+\epsilon q \in \operatorname{int}(C)$
for all real $\epsilon >0$. 

(3) This is immediate from (2), by Corollary \ref{1.2}.
\end{proof}

Here is more folklore concerning cones with non-empty interior:

\begin{prop}
\label{y}
Suppose that $C$ is a cone of $V$, $q$ is an interior point of $C$, and $v\in V$.
Then the following are equivalent:
\begin{enumerate}
\item $v$ is an interior point of $C$,
\item there exist $\epsilon >0$ such that $v- \epsilon q \in C$,
\item for every nonzero $L \in C^{\vee}$, $L(v)>0$.
\end{enumerate}
\end{prop}

\begin{proof}
(1) implies (2) by the easy direction of assertion (1) in Proposition \ref{x}.
To prove that (2) implies (3), pick $L \in C^{\vee}$ and $w \in V$
such that $L(w) \ne 0$. Since $q$ is an interior point of $C$,
there exists a $\delta>0$ such that $q \pm \delta w \in C$.
It follows that $L(q)\ge \delta \vert L(w) \vert >0$. Hence,
$L(v) \ge \epsilon L(q) >0$. Finally, we prove that (3) implies (1)
by contradiction. Note that $\operatorname{int}(C)$ is an open
convex set. If $v \not\in \operatorname{int}(C)$, there exists by the Separation Theorem a functional $L$ on $V$ such that $L(v) \le 0$
and $L(\operatorname{int}(C)) > 0$. It follows that $L(\overline{\operatorname{int}(C)}) \ge 0$.
But $\overline{\operatorname{int}(C)}=\overline{C}$ by assertion (3) of Proposition \ref{x},
hence $L(C) \ge 0$.
\end{proof}

We are also interested in cones satisfying $C\cup -C = V$. Note: For any cone $C$ of $V$, $C\cap -C$ is a subspace of $V$.

\begin{prop}\label{1.3} Let
$C$ be a cone of $V$ satisfying $C\cup -C = V$. The following are equivalent:
\begin{enumerate}
\item $C$ is closed in $V$.
\item The vector space $\frac{V}{C\cap -C}$ has dimension $\le 1$. 
\end{enumerate}
\end{prop}

\begin{proof} (2) $\Rightarrow$ (1). Replacing $V$ by $V/(C\cap -C)$ and $C$ by $C/(C\cap -C)$, we are reduced to the
case $C\cap -C= \{ 0 \}$. If $V$ is $0$-dimensional then $V = \{ 0 \} =
C$, so $C$ is closed in $V$. If $V$ is $1$-dimensional, fix $v \in C$,
$v\ne 0$. Then $V= \Bbb{R}v$ and $C= \Bbb{R}^+v$, so $C$ is closed
in $V$. (1) $\Rightarrow$ (2). Suppose $C$ is closed and
$\frac{V}{C\cap -C}$ has dimension $\ge 2$. Fix $v_1,v_2 \in V$ linearly
independent modulo $C\cap -C$. Let $W$ be denote the subspace of
$V$ spanned by $v_1,v_2$. Then $C\cap W$ is closed in $W$, $(C\cap
W)\cup -(C\cap W) = W$, and $v_1,v_2 \in W$ are linearly
independent modulo $(C\cap W)\cap -(C\cap W)$. In this way,
replacing $V$ by $W$ and $C$ by $C\cap W$, we are reduced to the
case where $V = \Bbb{R}v_1\oplus \Bbb{R}v_2$. Replacing $v_i$ by
$-v_i$, if necessary, we can suppose $v_i\in -C$, $i=1,2$. Then
$v:=v_1+v_2$ is an interior point of $-C$. In particular,
$\operatorname{int}(-C)\ne \emptyset$. Since $v_1$ and $v_2$ are
linearly independent modulo $C\cap -C$, we find $C\cap -C=\{0\}$
and $C\cap \operatorname{int}(-C) = \emptyset$. By the Separation
Theorem, there exists a linear map $L : V\rightarrow \Bbb{R}$ with
$L\ge 0$ on $C$, $L<0$ on $\operatorname{int}(-C)$ (so $L\le 0$ on
$-C$). Since $V$ is $2$-dimensional, there exists $w \in V$, $L(w)=0$, $w\ne
0$. Replacing $w$ by $-w$ if necessary, we may assume $w\in -C$
(so $w\notin C$). Consider the line through $v$ and $w$. Since
$L(v)<0$ and $L(w)=0$, there are points $u$ on this line
arbitrarily close to $w$ satisfying $L(u)>0$ (so $u\in C$). This
proves $w \in \overline{C}$ for all such points $w$, so $C$ is not
closed, a contradiction.
\end{proof}

\begin{cor}\label{1.4} Suppose
 $C$ is a cone of $V$ satisfying $C\cup -C = V$. Then $C^{\ddagger}$ is closed, i.e., $\overline{C} = C^{\ddagger}$.
\end{cor}

\begin{proof} According to Proposition \ref{1.3} it suffices to show that $\frac{V}{C^{\ddagger}\cap -C^{\ddagger}}$ has dimension at most one.
 Suppose this is not the case, so we have $v_1,v_2 \in V$ linearly independent modulo $C^{\ddagger}\cap -C^{\ddagger}$. Let $W = \Bbb{R}v_1\oplus \Bbb{R}v_2$
  and consider the closed cone $\overline{C\cap W}$ in $W$. Since $\overline{C\cap W} \cup -\overline{C\cap W} = W$, Proposition \ref{1.3} applied to
  the cone $\overline{C\cap W}$ of $W$ implies that $v_1,v_2$ are linearly dependent modulo  $\overline{C\cap W} \cap -\overline{C\cap W}$. On the other
  hand, $\overline{C\cap W} \subseteq C^{\ddagger}$, so $\overline{C\cap W} \cap -\overline{C\cap W} \subseteq C^{\ddagger}\cap -C^{\ddagger}$. This contradicts
  the assumption that $v_1,v_2$ are linearly independent modulo $C^{\ddagger}\cap -C^{\ddagger}$.
\end{proof}

\section{Closures of Quadratic Modules}

We introduce basic terminology, also see \cite{M} or \cite{Pr}.
Let $A$ be a commutative ring with $1$. For the rest of this work
we assume $\frac12\in A$. For $f_1,\dots,f_t \in A$,
$(f_1,\dots,f_t)$ denotes the ideal of $A$ generated by
$f_1,\dots,f_t$. For any prime ideal $\frak{p}$ of $A$,
$\kappa(\frak{p})$ denotes the residue field of $A$ at $\frak{p}$,
i.e., $\kappa(\frak{p}) $ is the field of fractions of the
integral domain $\frac{A}{\frak{p}}$. 
We denote by $\dim(A)$ the krull dimension of the ring $A$.

A \it quadratic module \rm of $A$ is a subset $Q$ of $A$
satisfying $Q+Q\subseteq Q$, $f^2Q\subseteq Q$ for all $f\in A$
and $1\in Q$.
If $Q$ is a quadratic module of $A$, then
$Q\cap -Q$ in an ideal of $A$ (since $\frac12\in A$). $Q\cap -Q$
is referred to as the \it support \rm of $Q$. The quadratic module
$Q$ is said to be \it proper \rm if $Q\neq A$. Since $\frac12\in
A$, this is equivalent to $-1 \notin Q$ (using the identity
$a=(\frac{a+1}{2})^2- (\frac{a-1}{2})^2$). A \it semiordering \rm
of $A$ is a quadratic module $Q$ of $A$ satisfying $Q\cup -Q = A$
and $Q\cap -Q$ is a prime ideal of $A$. A \it preordering \rm
(resp., \it ordering\rm) of $A$ is a quadratic module (resp.,
semiordering) of $A$ which is closed under multiplication. $\sum
A^2$ denotes the set of (finite) sums of squares of elements of
$A$.

We assume always that our ring $A$ is an $\Bbb{R}$-algebra. Then
$A$ comes equipped with the topology described in Section 1. Any
quadratic module $Q$ of $A$ is a cone, so $Q^{\ddagger}$ and
$\overline{Q}$ are cones. But actually, if $Q$ is a quadratic
module (resp., preordering) of $A$, then $Q^{\ddagger}$ and
$\overline{Q}$ is a quadratic module (resp. preordering) of $A$.
For $Q^\ddagger$ this is easy to see, for $\overline{Q}$ it is
proven as in \cite[Lemma 1]{CMN}.

In case $A$ is finitely generated, say  $x_1,\dots,x_n$ generate
$A$ as an $\Bbb{R}$-algebra, then the set of monomials
$x_1^{d_1}\cdots x_n^{d_n}$ is countable and generates $A$ as a
vector space over $\Bbb{R}$. In that case, the multiplication of
$A$ is continuous. This is another way to prove that closures of
quadratic modules (preorderings) are again quadratic modules
(preorderings) in that case. We denote the polynomial ring
$\Bbb{R}[x_1,\dots,x_n]$ by $\Bbb{R}[\underline{x}]$ for short.

A  quadratic module $Q$ is said to be \it
archimedean \rm if for every $f\in A$ there is an integer $k\ge 0$
such that $k+f\in Q$.

\begin{prop} \label{2.0}
For any quadratic module $Q$ of $A$, the following are equivalent:
\begin{enumerate}
\item $Q$ is archimedean,
\item $1$ belongs to the interior of $Q$,
\item $Q$ has non-empty interior.
\end{enumerate}
\end{prop}

\begin{proof} Clearly, $Q$ is archimedean iff $1$ is
an algebraic interior point of $Q$, hence (1) $\Leftrightarrow$ (2)
follows from the first assertion of Proposition \ref{x}.
It remains to show (3) $\Rightarrow$ (2).
Every functional $L \in Q^{\vee}$ satisfies the
Cauchy-Schwartz inequality, $L(a)^2 \le L(1)L(a^2)$.
If follows that every nonzero $L \in Q^{\vee}$
satisfies $L(1)>0$. Since $Q$ has non-empty interior,
it follows by Proposition \ref{y} that
$1$ is an interior point of $Q$.
\end{proof}

The simplest example of a non-archimedean quadratic module
is the quadratic module $Q=\sum \RR[x]^2$ of the algebra $A=\RR[x]$. By Proposition \ref{2.0}, $1$ is not an interior point of $Q$ and, by its proof, $L(1)>0$ for
every nonzero $L \in Q^{\vee}$. So, the
implication (3) $\Rightarrow$ (1) of Proposition \ref{y}
is not valid in general.

\begin{prop}\label{2.1} Let $Q$ be a semiordering of $A$. If $Q$ is not
archi\-medean then $Q^{\ddagger} = \overline{Q} = A$. If $Q$
is archimedean, then there exists a unique ring homomorphism
$\alpha : A \rightarrow \Bbb{R}$ with $Q \subseteq
\alpha^{-1}(\Bbb{R}^+)$, and $Q^{\ddagger} =\overline{Q} =
\alpha^{-1}(\Bbb{R}^+)$.
\end{prop}

\begin{proof} If $Q$ is not archimedean there exists $q\in A$ with $k+q \notin Q$ for all real $k>0$. Then $-k-q \in Q$, i.e.,
 $-1-\frac{1}{k}q \in Q$, for all real $k>0$. This proves $-1 \in Q^{\ddagger}$, so $Q^{\ddagger}=\overline{Q} = A$.

Suppose $Q$ is archimedean. According to \cite[Theorem 5.2.5]{M}
there exists a ring homomorphism $\alpha : A \rightarrow \Bbb{R}$
such that $Q \subseteq \alpha^{-1}(\Bbb{R}^+)$. $\alpha$ is linear
so $\alpha^{-1}(\Bbb{R}^+)$ is closed, so $\overline{Q} \subseteq
\alpha^{-1}(\Bbb{R}^+)$. If $f \in \alpha^{-1}(\Bbb{R}^+)$ then
for any real $\epsilon >0$, $\alpha(f+\epsilon) >0$ so $f+\epsilon
\in Q$. (If $f+\epsilon \notin Q$ then $-(f+\epsilon) \in Q$ so
$-(f+\epsilon) \in \alpha^{-1}(\Bbb{R}^+)$, which contradicts our
assumption.)  It follows that $Q^{\ddagger}=\overline{Q} =
\alpha^{-1}(\Bbb{R}^+)$. (This can also be deduced from
Proposition \ref{1.3}.) Uniqueness of $\alpha$ is for example
\cite[Lemma 5.2.6]{M}.
\end{proof}

\begin{prop}\label{2.2} Let $A$ be finitely generated. For any set of semiorderings $\mathcal{Y}$ of $A$,
$$(\cap_{Q\in \mathcal{Y}}Q)^{\ddagger}=\overline{\cap_{Q\in \mathcal{Y}} Q} = \cap_{Q\in \mathcal{Y}} \overline{Q}.$$
\end{prop}

\begin{proof} Suppose  $f \in \cap_{Q\in \mathcal{Y}} \overline{Q}$.
Fix generators $x_1,\dots,x_n$ of $A$ as an $\Bbb{R}$-algebra
and let $d$ denote the degree of $f$ viewed as a polynomial in
$x_1,\dots,x_n$ with coefficients in $\Bbb{R}$. Let $g =
1+\sum_{i=1}^n x_i^2$ and fix an integer $e$ with $2e>d$. We claim
that for any real $\epsilon >0$ and any $Q \in \mathcal{Y}$,
$f+\epsilon g^e \in Q$. This will prove that $f+\epsilon g^e \in
\cap_{Q\in \mathcal{Y}} Q$ for any real $\epsilon >0$, so $f\in
(\cap_{Q\in \mathcal{Y}} Q)^{\ddagger}$, which will complete the
proof. Let $\frak{p} := Q\cap -Q$, let $Q'$ denote the extension
of $Q$ to the residue field $\kappa(\frak{p})$, and let $v$ denote
the natural valuation of $\kappa(\frak{p})$ associated to $Q'$
(e.g., see \cite[Theorem 5.3.3]{M}). To prove the claim we
consider two cases. Suppose first that $v(x_i+\frak{p})<0$ for
some $i$. Reindexing we may suppose $v(x_1+\frak{p}) \le
v(x_i+\frak{p})$ for all $i$. Then $v(g^e+\frak{p}) =
ev(g+\frak{p}) = 2ev(x_1+\frak{p}) < dv(x_1+\frak{p}) \le
v(f+\frak{p})$. It follows that the sign of $f+\epsilon g^e$ at
$Q$ is the same as the sign of $g^e$ at $Q$ in this case, i.e.,
$f+\epsilon g^e \in Q$. In the remaining case $v(x_i+\frak{p}) \ge
0$ for all $i$ so $\frac{A}{\frak{p}}$ is a subring of the
valuation ring $B_v$ in this case. Since the residue field of $v$
is $\Bbb{R}$, we have a ring homomorphism $\alpha : A \rightarrow
\Bbb{R}$ defined by the composition $A \rightarrow
\frac{A}{\frak{p}} \subseteq B_v \rightarrow \Bbb{R}$. Then
$\overline{Q} \subseteq \alpha^{-1}(\Bbb{R}^+)$ so $\alpha(f)\ge
0$ and $\alpha(f+\epsilon g^e)>0$. This implies that $f+\epsilon
g^e \in Q$ also holds in this case.
\end{proof}

We assume always that $M$ is a quadratic module of $A$. For some
results we need that $M$ and/or $A$ are finitely generated, some
results hold in general. Let $\mathcal{Y}_M$ denote the set of all
semiorderings of $A$ containing $M$, $\mathcal{X}_M$ the set of
all orderings of $A$ containing $M$ and $\mathcal{K}_M$ the set of
geometric points of $\mathcal{X}_M$, i.e., the orderings of $A$
having the form $\alpha^{-1}(\Bbb{R}^+)$ for some ring
homomorphism $\alpha: A \rightarrow \Bbb{R}$ with $M\subseteq
\alpha^{-1}(\Bbb{R}^+)$. For any set of semiorderings
$\mathcal{Y}$ of $A$, define $\operatorname{Pos}(\mathcal{Y}) :=
\cap_{Q\in \mathcal{Y}}Q$, i.e., $\operatorname{Pos}(\mathcal{Y})
:= \{ f \in A \mid f\ge 0 \text{ on } \mathcal{Y}\}$. Since
$\mathcal{K}_M \subseteq \mathcal{X}_M \subseteq \mathcal{Y}_M$ it
follows that
\begin{equation}\label{eqn1}\operatorname{Pos}(\mathcal{K}_M) \supseteq
\operatorname{Pos}(\mathcal{X}_M) \supseteq
\operatorname{Pos}(\mathcal{Y}_M)\supseteq M.
\end{equation}

\begin{prop}\label{2.3} Let $A$ be finitely generated, $M$ an  arbitrary quadratic module in $A$. Then
$$\operatorname{Pos}(\mathcal{K}_M)= \operatorname{Pos}(\mathcal{Y}_M)^{\ddagger}=
\overline{\operatorname{Pos}(\mathcal{Y}_M)}.$$
\end{prop}

\begin{proof} Immediate from Proposition \ref{2.1} and \ref{2.2}.
\end{proof}

One can improve on (\ref{eqn1}) and Proposition \ref{2.3} in important 
 cases:

\begin{prop}\label{2.4} \
\begin{enumerate}
\item If $A$ and $M$ are finitely generated, then  $\operatorname{Pos}(\mathcal{K}_M) =
\operatorname{Pos}(\mathcal{X}_M).$
\item If either $M$ is a preordering in $A$, or $A$ is finitely generated and $\dim(\frac{A}{M\cap-M}) \le 1$,
then $\operatorname{Pos}(\mathcal{X}_M) = \operatorname{Pos}(\mathcal{Y}_M)$.
\end{enumerate}

\end{prop}

\begin{proof} (1) is immediate from Tarski's Transfer Principle.

(2) If $\dim(\frac{A}{M\cap-M}) \le 1$ then every semiordering lying over $M$ is an ordering,
e.g., by \cite[Theorem 7.4.1]{M}, so the result is clear in this case. Suppose now that $M$ is a preordering,
 $f\ge 0$ on $\mathcal{X}_M$ and $Q \in \mathcal{Y}_M$. Let $\frak{p} := Q\cap -Q$, $M' :=$ the extension of $M$
 to $\kappa(\frak{p})$. $M'$ is a preordering of $\kappa(\frak{p})$ so it is the intersection of the orderings of
 $\kappa(\frak{p})$ lying over $M'$, by the Artin-Schreier Theorem \cite[Lemma 1.4.4]{M}. Since $f\ge 0$ on $\mathcal{X}_M$
 this forces  $f+\frak{p} \in M'$. Since $M'$ is a subset of the extension of $Q$ to $\kappa(\frak{p})$, this implies in turn that $f \in Q$.
\end{proof}

$\operatorname{Pos}(\mathcal{Y}_M)$ can also be described in other ways, which make no explicit mention of $\mathcal{Y}_M$:
\begin{align*} \operatorname{Pos}(\mathcal{Y}_M) =  \{ f \in A \mid & pf = f^{2m}+q \text{ for some } p \in \sum A^2, q \in M, m\ge 0\} \\ =
\{ f\in A \mid & f+\frak{p} \text{ belongs to the extension of } M \text{ to } \kappa(\frak{p}) \\ &\forall \text{ primes } \frak{p} \text{ of } A\}.\end{align*}
This is well-known and is a consequence of the abstract Positivstellensatz for semiorderings, e.g., see \cite{J} or \cite[Theorem 5.3.2]{M}.
Typically one uses ideas from quadratic form theory and valuation theory to decide when $f+\frak{p}$ lies in the extension of $M$ to $\kappa(\frak{p})$;
 see \cite{JP} and \cite{M}. Note that one needs only consider primes $\frak{p}$ satisfying $f\notin \frak{p}$ and $(M+\frak{p})\cap -(M+\frak{p}) = \frak{p}$.

We turn now to $\overline{M}$. One has the obvious commutative diagram:

\[
\xymatrix{ \overline{M} \ar@{->}
[r] & \operatorname{Pos}(\mathcal{K}_M) \\ M
\ar@{->}
[r] \ar@{->} [u] & \operatorname{Pos}(\mathcal{Y}_M), \ar@{->}
[u] }
\]
The arrows here denote inclusions.
Interest in $\overline{M}$ stems from the Moment Problem:

\begin{prop}\label{2.5} Let $A$ be finitely generated and $M$ an arbitrary quadratic module of $A$. Then the following are equivalent:
\begin{enumerate}
\item $\overline{M} = \operatorname{Pos}(\mathcal{K}_M)$.
\item For each $L \in M^{\vee}$ there exists a positive Borel measure $\mu$ on $\mathcal{K}_M$ such that $L(f) = \int f \, d\mu$ for all $f \in A$.
\end{enumerate}
\end{prop}

\begin{proof} This follows from Haviland's Theorem \cite[Theorem 3.1.2]{M}, using  $\overline{M}=M^{\vee\vee}$, as explained above.
\end{proof}

See \cite[Theorem 3.2.2]{M} for an extended version of Haviland's
Theorem. For arbitrary $A$ and $M$, we say $M$ satisfies the \it
strong moment property \rm (SMP) if condition (1) of Proposition
\ref{2.5} holds.

 In computing $\overline{M}$ it seems there are only two basic tools available, which are the following two theorems.

\begin{thm}\label{2.6} Let $A$ and $M$ be finitely generated. If $M$ is stable then
$M^{\ddagger}=\overline{M} = M+\sqrt{M\cap -M}$ and $\overline{M}$ is stable.
\end{thm}

See \cite{Sc} or \cite[Theorem 4.1.2]{M} for the proof of Theorem \ref{2.6}. Here, $\sqrt{M\cap -M}$ denotes the radical of the ideal $M\cap -M$. Recall:
 $M$ is said to be \it stable \rm \cite{N2} \cite{PS} \cite{Sc} if for each finite dimensional subspace $V$ of $A$ there exists a finite dimensional
  subspace $W$ of $A$ such that each $f\in M\cap V$ is expressible as $f = \sigma_0+\sigma_1g_1+\dots+\sigma_sg_s$ where $g_1,\dots,g_s$ are the fixed
  generators of $M$ and the $\sigma_i$ are sums of squares of elements of $W$. See \cite{M} for an equivalent definition.

Interest in stability arose in the search for examples where (SMP)
fails. The quadratic module $\sum \Bbb{R}[\underline{x}]^2$ of the
polynomial ring $\Bbb{R}[\underline{x}]$ is stable. Theorem
\ref{2.6} was proved first in this special case in \cite{BCJ}, and
the result was then used to show that $\sum
\Bbb{R}[\underline{x}]^2$ does not satisfy (SMP) if $n\ge 2$. More
recently, in \cite[Theorem 5.4]{Sc}, it is shown that if $M$ is
stable and $\dim(\mathcal{K}_M)\ge 2$ then $M$ does not satisfy
(SMP). See \cite{BCJ} \cite{KM} \cite{M} \cite{N2} \cite{Pl}
 \cite{PS} for examples where stability holds.

The second basic tool is the following result, which is both a strengthening and an extension to quadratic modules of Schm\"udgen's
fibre theorem in \cite{S2}; also see \cite{N1}.  

\begin{thm}\label{new} Let $f \in A$, $a,b \in \Bbb{R}$.
\begin{enumerate}
\item If $a < f < b$ on $\mathcal{Y}_M$ then $b-f,f-a \in M^{\ddagger}$.
\item If $A$ has countable vector space dimension and $b-f,f-a \in \overline{M}$ then
$\overline{M} = \cap_{a\le \lambda \le b}
\overline{M_{\lambda}}$, where $M_{\lambda} := M+(f-\lambda)$.
\end{enumerate}
\end{thm}

From part (1) one can immediately deduce that $b'-f, f-a'\in
\overline{M}$ where $a':= \sup\{ a\in \Bbb{R} \mid a\le f \text{
on } \mathcal{Y}_M\}$, $b':= \inf\{ b \in \Bbb{R} \mid f\le b
\text{ on } \mathcal{Y}_M\}$. In fact one even gets  $b'-f,
f-a'\in (M^{\ddagger})^{\ddagger}$, the second cone in the sequence of iterated
sequential closures of $M$.

Part (1) is useful in conjunction with part (2).
If $M$  and $A$ are finitely generated and
either $M$ is a preordering or $\dim(\frac{A}{M\cap -M}) \le 1$,
then the assumption  that $a\le f \le b$ on $\mathcal{Y}_M$ is equivalent to the assumption that
$a\le f \le b$ on $\mathcal{K}_M$; see Proposition \ref{2.4}. In particular, 
parts (1) and (2) taken together
yield Schm\"udgen's result in \cite{S2} as a special case.

Part (1) is also of independent interest. It is an improvement of the corresponding result in \cite{S2}, not only because of the extension from preorderings to quadratic modules, but also because the conclusion $b-f, f-a \in \overline{M}$ has been replaced by the stronger conclusion  $b-f, f-a \in M^{\ddagger}$.


The proof of (2) for finitely generated algebras and
finitely generated quadratic modules is given already in \cite[Theorem
4.4.1]{M}.  The general case of an algebra of countable vector
space dimension and arbitrary $M$ is almost the same, see
\cite[Theorem 2.6]{N4}.

As explained already in  \cite{M} \cite{S1} \cite{S2}, to prove (1), one is reduced to showing:

\begin{lemma}\label{2.8} Suppose $f\in A$, $\ell\in \Bbb{R}$, $\ell^2-f^2 >0$ on $\mathcal{Y}_M$. Then $\ell^2-f^2 \in M^{\ddagger}$.
\end{lemma}

\begin{proof}
By the abstract Positivstellensatz for semiorderings, see
\cite[Theorem 5.3.2]{M}, the hypothesis implies
$(\ell^2-f^2)p=1+q$ for some $p\in \sum A^2$, $q\in M$.  Now one
starts with Schm\"udgen's argument  involving Hamburger's Theorem
(also see the proof of \cite[Theorem 3.5.1]{M}), i.e. one proceeds
as follows:

\smallskip
Claim 1: $\ell^{2i}p-f^{2i}p \in M$ for all $i\ge 1$. Since
$\ell^2p-f^2p = (\ell^2-f^2)p= 1+q$, this is clear when $i=1$.
Since $$\ell^{2i+2}p-f^{2i+2}p =
\ell^2(\ell^{2i}p-f^{2i}p)+f^{2i}(\ell^2p-f^2p),$$ the result
follows, by induction on $i$.

\smallskip
Claim 2: $\ell^{2i+2}p-f^{2i} \in M$ for all $i\ge 1$. Since
$$\ell^{2i+2}p-f^{2i} =
\ell^2(\ell^{2i}p-f^{2i}p)+f^{2i}(\ell^2p-1),$$  and $\ell^2p-1 =
q+f^2p \in M$, this follows from Claim 1.

\smallskip Now we use a little technical trick. Define $V:=\Bbb{R}[f] +\Bbb{R}p$, a vector subspace of $A$.
Write $M_V:=M\cap V$, so $M_V$ is a cone in $V$. We claim
that $p+1$ is an interior point of $M_V$ in $V$. Indeed,
$$\ell^{2i+2}p \pm 2f^i+1=\ell^{2i+2}p-f^{2i} + (f^i\pm1)^2\in
M_V$$for all $i\geq 1$, using Claim 2. So with $N:=\max\{1,\ell\}$
we have for every $i\geq 1$ $$(p+1) \pm \frac{2}{N^{2i+2}}\cdot
f^i\in M_V.$$ Clearly also
$$(p+1)\pm 1 \in M_V \mbox{ and } (p+1)\pm p\in M_V$$ holds, which
proves the claim, using Proposition \ref{x}.

We now claim that $\ell^2-f^2$ belongs to
$\overline{M_V}=(M_V)^{\vee\vee}$ in $V$. Therefore fix $L \in (M_V)^{\vee}$ and consider the linear map $L_1 : \Bbb{R}[Y] \rightarrow \Bbb{R}$
defined by $L_1(r(Y)) = L(r(f))$. Here, $r(f)$ denotes the image of $r(Y)$ under the algebra homomorphism from $\Bbb{R}[Y]$ to $V$
 defined by $Y \mapsto f$. Since $r(f)^2$ is a square in $A$, and $M$ contains all squares, and $L$ is $\ge 0$ on $M_V$,
 we see that $L_1(r^2) = L(r(f)^2)\ge 0$ for all $r \in \Bbb{R}[Y]$. By Hamburger's Theorem \cite[Corollary 3.1.4]{M}, there
 exists a Borel measure $\nu$ on $\Bbb{R}$ such that $$L(r(f))= L_1(r) = \int r \, d\nu,$$ for each $r\in \Bbb{R}[Y]$. Let $\lambda >0$
 and let $\mathcal{X}_{\lambda}$ denote the characteristic function of the set $(-\infty, -\lambda)\cup (\lambda, \infty)$.
 Then $$\lambda^{2i}\int \mathcal{X}_{\lambda}\, d\nu \le \int Y^{2i}\, d\nu = L_1(Y^{2i}) = L(f^{2i}) \le \ell^{2i+2}L(p).$$
 The first inequality follows from the fact that $\lambda^{2i}\mathcal{X}_{\lambda}\le Y^{2i}$ on $\Bbb{R}$. The last inequality
 follows from Claim 2. Since this holds for any $i\ge 1$, it clearly implies that $\int \mathcal{X}_{\lambda}\, d\nu = 0$, for any $\lambda >\ell$.
 This implies, in turn, that $\int \mathcal{X}_{\ell}\, d\nu = 0$ i.e., the set $(-\infty,-\ell)\cup (\ell, \infty)$ has $\nu$ measure zero.
 Since $Y^2\le \ell^2$ holds on the interval $[-\ell,\ell]$, this yields $$L(f^2)= \int Y^2\, d\nu \le \int \ell^2 \, d\nu = L(\ell^2).$$
 This proves $L(\ell^2-f^2)\ge 0$. Since this is true for any $L \in (M_V)^{\vee}$, this proves $\ell^2-f^2\in (M_V)^{\vee\vee}
 =\overline{M_V}$.

 Now finally, since $M_V$ has an interior point in $V$, $\overline{M_V}=(M_V)^{\ddagger}$ by Proposition \ref{x}.
 Therefore, $\ell^2-f^2\in  M_V^{\ddagger} \subseteq M^{\ddagger}.$
\end{proof}

Theorem \ref{new} can be used to produce examples where (SMP)
holds, see \cite{M} \cite{N1} \cite{S1} \cite{S2}. Assuming the
hypothesis of Theorem \ref{new} (2), $M$ satisfies (SMP) iff each
$M_{\lambda}$ satisfies (SMP). The implication ($\Leftarrow$) is
immediate from Theorem \ref{new} (2). The implication ($\Rightarrow$)
is a consequence of the following:

\begin{lemma}\label{2.9} If $A$ is finitely generated and $M$ satisfies (SMP) then so does $M+I$, for each ideal $I$ of $A$.
\end{lemma}


See \cite[Proposition 4.8]{Sc} for the proof of Lemma \ref{2.9}.
Theorem \ref{new} has also been used to construct an example where
$M^{\ddagger} \ne \overline{M}$; see \cite{N3}.   The reader will
encounter additional applications of Theorem \ref{new} in Sections 3 and 4.

\section{The Compact Case}

We recall basic facts concerning archimedean quadratic modules. We characterize archimedean quadratic modules in various ways.

\begin{thm}\label{3.1} Suppose $M$ is archimedean. Then $f\ge 0$ on $\mathcal{K}_M$ $\Rightarrow$ $f+\epsilon \in M$ for all
real $\epsilon >0$. In particular, $M^{\ddagger}=\overline{M}=
\operatorname{Pos}(\mathcal{K}_M)$.
\end{thm}

Theorem \ref{3.1} is Jacobi's Representation Theorem \cite{J}. See
\cite[Theorem 5.4.4]{M} for an elementary proof. There is no
requirement that $A$ or $M$ be finitely generated. We give
another proof of Theorem \ref{3.1}, based on Theorem \ref{new} (1). 

\begin{proof} Suppose $f \in A$, $f\ge 0$ on $\mathcal{K}_M$, $\epsilon \in \Bbb{R}$, $\epsilon >0$. For each $Q \in \mathcal{Y}_M$, $Q$ is archimedean (because $M$ is) so, arguing as in the proof of Proposition \ref{2.1}, $\exists$ a ring homomorphism $\alpha : A \rightarrow \Bbb{R}$ such that $\alpha^{-1}(\Bbb{R}^+) \supseteq Q$ and $f+\epsilon \in Q$. This proves $f\ge -\epsilon$ on $\mathcal{Y}_M$.
Since $M$ is archimedean, $\exists$ $b\in \Bbb{R}$, $b-f\in M$, so
$b\ge f \ge -\epsilon$ on $\mathcal{Y}_M$. According to Theorem
\ref{new} (1) this implies $f+\epsilon \in \overline{M}$ for each real
$\epsilon >0$, so $f\in \overline{M}$. Since $M$ is archimedean,
$1$ is an algebraic
 interior point of $M$. By Proposition \ref{x}, $f+\epsilon \in M$ for all real $\epsilon>0$.
\end{proof}

The following result is proved in \cite[Theorem 5.1.18]{Pr}:

\begin{thm}\label{b} If $M$ is archimedean then every maximal semiordering $Q$ of $A$ lying over $M$ is archimedean.
If $A$ is a finitely generated $\Bbb{R}$-algebra the converse is also true.
\end{thm}
There is no requirement here that $M$ be finitely generated. Note:
Maximal  semiorderings and maximal proper quadratic modules are
the same thing, e.g., see \cite{J} or \cite[Sect. 5.3]{M}. By
\cite[Theorem 5.2.5]{M}, every maximal semiordering $Q$ which is
archimedean has the form $Q= \alpha^{-1}(\Bbb{R}^+)$ for some
(unique) ring homomorphism $\alpha : A \rightarrow \Bbb{R}$.


\begin{cor}\label{3.2} Suppose $x_1,\dots,x_n$ generate $A$ as an $\Bbb{R}$-algebra. The following are equivalent:
\begin{enumerate}
\item $M$ is archimedean.
\item $\sum_{i=1}^n x_i^2$ is bounded on $\mathcal{Y}_M$.
\end{enumerate}
\end{cor}

If $M$ is a finitely generated preordering then, by Proposition \ref{2.4}, $\sum
x_i^2$ is bounded  on $\mathcal{Y}_M$ $\Leftrightarrow$ $\sum
x_i^2$ is bounded on $\mathcal{K}_M$ $\Leftrightarrow$
$\mathcal{K}_M$ is compact. In this case, Corollary \ref{3.2} is
just ``W\"ormann's Trick''; see \cite{M} \cite{W}. 

\begin{proof} (1) $\Rightarrow$ (2) is clear. (2) $\Rightarrow$ (1). Fix a positive constant $k$ such that $k-\sum x_i^2 > 0$ on
$\mathcal{Y}_M$. By \cite[Corollary 5.2.4]{M}, each maximal semiordering $Q$ of $A$ lying over $M$ is archimedean.
Now apply Theorem \ref{b}.
\end{proof}

The second assertion of Theorem \ref{b} is not true for general $A$. In \cite{M0} an example is given of a countably infinite dimensional $\Bbb{R}$-algebra $A$ such that every maximal proper quadratic module $Q$ of $A$ is archimedean (so has the form $\alpha^{-1}(\Bbb{R}^+)$ for some ring homomorphism $\alpha : A \rightarrow \Bbb{R}$), but $\sum A^2$ itself is not archimedean. In fact, in this example, the only elements $h\in A$ satisfying $\ell\pm h \in \sum A^2$ for some integer $\ell\ge 1$ are the elements of $\Bbb{R}$.
But there is a certain weak version of the second assertion of Theorem \ref{b} which does hold for general $A$:

\begin{thm}\label{c} If every maximal semiordering of $A$ lying over $M$ is archimedean, then $\mathcal{K}_M$ is compact
and $(M^{\ddagger})^{\ddagger} = \overline{M} = \operatorname{Pos}(\mathcal{K}_M)$. In particular,
$(M^{\ddagger})^{\ddagger}$ is archimedean.
\end{thm}

There is no requirement here that $A$ or $M$ be finitely generated.

\begin{proof} The result follows from Theorem \ref{new} (1) once we prove that $\mathcal{K}_M$
  is compact (using the fact that $f \ge 0$ on $\mathcal{K}_M$ $\Rightarrow$ $f > -\epsilon$ on $\mathcal{Y}_M$,
  for all real $\epsilon > 0$). Fix $f\in A$ and let $M_{\ell} = M-\sum A^2(\ell^2-f^2)$. Then $M_{\ell} \subseteq M_{\ell+1}$.
  If $-1 \notin \cup_{\ell\ge 1} M_{\ell}$ then we would have a maximal semiordering $Q$ containing $\cup_{\ell\ge 1} M_{\ell}$.
  Then $-(\ell^2-f^2)\in Q$ for all $\ell \ge 1$, so $(\ell-1)^2-f^2 \notin Q$ for all $\ell \ge 1$. This is a contradiction.
  Thus $-1 = s-p(\ell^2-f^2)$ for some $s\in M$, $p\in \sum A^2$ and some integer $\ell \ge 1$. This implies $-\ell < \alpha(f)< \ell$
  for all $\alpha \in \mathcal{K}_M$, for some integer $\ell \ge 1$ (depending on $f$), say $\ell = \ell_f$. Then $\mathcal{K}_M$ is identified
  with a closed subspace of the compact space $\prod_{f\in A}[-\ell_f,\ell_f]$.
\end{proof}

Note: Instead of arguing with the quadratic modules $M_{\ell}$, one could exploit the compactness of the spectral space $\operatorname{Semi-Sper}(A)$, as was done in the proof of \cite[Theorem 5.1.18]{Pr}. This shows that if $\mathcal{Y}$ is any set of archimedean semiorderings in $\operatorname{Semi-Sper}(A)$ which is closed in the constructible topology then $\cap_{Q \in \mathcal{Y}} Q$ is archimedean.

\begin{cor}\label{A} The following are equivalent:
\begin{enumerate}
\item $\overline{M}$ is archimedean.
\item Every maximal semiordering of $A$ lying over $\overline{M}$ is archimedean.
\item $\mathcal{K}_M$ is compact and $\overline{M} = \operatorname{Pos}(\mathcal{K}_M)$.
\end{enumerate}
\end{cor}

\begin{proof} (1) $\Rightarrow$ (2) and (3) $\Rightarrow$ (1) are obvious. If $\alpha \in \mathcal{K}_M$ then $\alpha^{-1}(\Bbb{R}^+)$ is closed and $M \subseteq \alpha^{-1}(\Bbb{R}^+)$, so $\overline{M} \subseteq \alpha^{-1}(\Bbb{R}^+)$. This proves that $\mathcal{K}_M = \mathcal{K}_{\overline{M}}$. The implication (2) $\Rightarrow$ (3) follows from this observation, by applying Theorem \ref{c} to the quadratic module $N = \overline{M}$.
\end{proof}

Note: Since $\mathcal{K}_M = \mathcal{K}_{\overline{M}}$, one sees now that Corollary \ref{A} is not really a statement about the quadratic module $M$, but rather it is a statement about the closed quadratic module  $\overline{M}$.

Clearly $M$ archimedean $\Rightarrow$ $\overline{M}$ archimedean $\Rightarrow$ $\mathcal{K}_M$ compact.
We conclude by giving concrete examples to show that $\mathcal{K}_M$ compact $\not\Rightarrow$ $\overline{M}$ archimedean and $\overline{M}$ archimedean $\not\Rightarrow$ $M$ archimedean:

\begin{ex}\label{3.3} Let $A := \Bbb{R}[\underline{x}]$, $n\ge 2$, $M :=$ the quadratic module of $\Bbb{R}[\underline{x}]$ generated by $$x_1-1, \cdots, x_n-1, c-\prod_{i=1}^n x_i,$$ where $c$ is a positive real constant. Then $\mathcal{K}_M$ is compact (possibly empty, depending on the value of $c$), but, as explained in \cite{JP}, $M$ is not archimedean. As pointed out in \cite{N2} (also see \cite{M}) $M$ is also stable, so $\overline{M} = M$, by Theorem \ref{2.6}. 
\end{ex}

\begin{ex}\label{3.4} Let $A:= \Bbb{R}[\underline{x}]$, $n\ge 2$,
$M:=$ the quadratic module of $\Bbb{R}[\underline{x}]$ generated by
$$1-x_1,\dots,1-x_n, \prod_{i=1}^nx_i-c, x_1x_n^2, x_1x_2x_n^2,\dots, x_1\cdots x_{n-1}x_n^2,$$
where $c$ is a positive real constant. In this example, $\mathcal{K}_M$ is compact, $M$ is not archimedean,
but $\overline{M} = \operatorname{Pos}(\mathcal{K}_M)$, so $\overline{M}$ is archimedean. 
One checks that $0<x_1\le 1$ on $\mathcal{Y}_M$ so, by Theorem \ref{new},
$\overline{M} = \cap_{0\le \lambda \le 1} \overline{M_{\lambda}}$ where
$M_{\lambda} := M+(x_1-\lambda)$ so, to prove $M$ satisfies (SMP) it suffices to prove each
$M_{\lambda}$ satisfies (SMP). Exploiting the natural isomorphism
$\frac{\Bbb{R}[\underline{x}]}{(x_1-\lambda)} \cong \Bbb{R}[x_2,\dots,x_n]$, this reduces to showing that the quadratic module $N_{\lambda}$ of $\Bbb{R}[x_2,\dots,x_n]$ generated by $$1-\lambda,1-x_2,\dots, 1-x_n, \lambda\prod_{i=2}^n x_i-c,\lambda x_n^2,\lambda x_2x_n^2,\dots, \lambda x_2\dots x_{n-1}x_n^2$$ satisfies (SMP). If $\lambda =0$ then $-1\in N_{\lambda}$ so this is true for trivial reasons. If $0<\lambda \le 1$ then $N_{\lambda}$ is generated by $$1-x_2,\dots,1-x_n, \prod_{i=2}^n x_i-\frac{c}{\lambda},x_2x_n^2,\dots, x_2\dots x_{n-1}x_n^2,$$ and $N_{\lambda}$ satisfies (SMP) by induction on $n$. This proves $M$ satisfies (SMP). To show $M$ is not archimedean it suffices to show $k^2-x_1^2 \notin M$ for each real $k$. Taking $x_2=\dots = x_{n-1} = 1$, this reduces to the case $n=2$ and, in this case, it can be verified by an easy degree argument (considering terms of highest degree). But actually, one can say more. Using a valuation-theoretic argument one can show that the only elements of $\Bbb{R}[\underline{x}]$ which are bounded on $\mathcal{Y}_M$ are the elements in $\Bbb{R}[x_1]$. Using this, one checks that the only elements $f$ of $\Bbb{R}[\underline{x}]$ satisfying $k^2-f^2 \in M$ for some real constant $k$ are the elements of $\Bbb{R}$.
\end{ex}

Note: There is a valuation-theoretic criterion for deciding when $M$ is archimedean, given that $\mathcal{K}_M$ is compact; see \cite{JP} or \cite{M}. But typically this does not apply to $\overline{M}$, because $\overline{M}$ is not finitely generated.

\section{Computation of $\overline{M}$ in Special Cases}

If $\dim(\frac{A}{M\cap -M}) \le 1$, Theorems \ref{2.6} and \ref{new}
combine to yield a recursive  description of $\overline{M}$. This
is a consequence of the following result:

\begin{thm}\label{4.1} Let $A$ be finitely generated and suppose the finitely generated quadratic module $M$ fulfills $\dim(\frac{A}{M\cap -M}) \le 1$.
 If the only elements of $A$ bounded
on $\mathcal{K}_M$ are the elements of $\Bbb{R} + M\cap -M$ then $M$ is stable.
\end{thm}

A preordering version of Theorem \ref{4.1} appears already in \cite[Corollary 2.11]{Pl}.

\begin{proof} Replacing $A$ by $\frac{A}{M\cap -M}$ and $M$ by $\frac{M}{M\cap -M}$, and applying \cite[Lemma 3.9]{Sc} or arguing as in \cite[Lemma 4.1.1]{M},
 we are reduced to the case $M\cap -M = \{ 0 \}$. Since $A$ is noetherian there are just finitely many minimal primes of $A$.
 Let $\frak{p}$ be a minimal prime of $A$, $\kappa(\frak{p}) := \operatorname{ff}(\frac{A}{\frak{p}})$. According to \cite[Proposition 2.1.7]{M},
  $(M+\frak{p})\cap -(M+\frak{p}) = \frak{p}$, i.e., $M$ extends to a proper preordering of $\kappa(\frak{p})$. $\dim(\frac{A}{\frak{p}})$ is
  either $0$ or $1$. For $\dim(\frac{A}{\frak{p}}) = 1$ let $S_{\infty,\frak{p}}$ denote the set of valuations $v\ne 0$ of $\kappa(\frak{p})$
  compatible with some ordering of $\kappa(\frak{p})$ lying over the extension of $M$ to $\kappa(\frak{p})$ and such that $\frac{A}{\frak{p}} \nsubseteq B_v$,
  where $B_v \subseteq \kappa(\frak{p})$ is the valuation ring of $v$. By Noether Normalization $\exists$ $t \in \frac{A}{\frak{p}}$ transcendental over $\Bbb{R}$
   such that $\frac{A}{\frak{p}}$ is integral over $\Bbb{R}[t]$. Since $B_v$ is integrally closed, $\frac{A}{\frak{p}} \nsubseteq B_v$ $\Leftrightarrow$ $t\notin B_v$
   $\Rightarrow$ $v$ is one of the extensions of the discrete valuation $v_{\infty}$ of $\Bbb{R}(t)$. Since $[\kappa(\frak{p}) : \Bbb{R}(t)] <\infty$, the set
   $S_{\infty, \frak{p}}$ is finite and each $v\in S_{\infty,\frak{p}}$ is discrete with residue field $\Bbb{R}$. Let $S_{\infty} :=$ the union of the sets
   $S_{\infty,\frak{p}}$, $\frak{p}$ running through the mimimal primes of $A$ with $\dim(\frac{A}{\frak{p}}) = 1$. Thus $S_{\infty}$ is finite. View elements
   of $S_{\infty}$ as functions $v : A \rightarrow \Bbb{Z}\cup \{ \infty\}$ by defining $v(f) = v(f+\frak{p})$ if $v \in S_{\infty,\frak{p}}$.

If $\mathcal{K}_{M+\frak{p}}$ is compact then every $f\in A$ is bounded on $\mathcal{K}_{M+\frak{p}}$ so either $\dim(\frac{A}{\frak{p}}) = 0$ or
$\dim(\frac{A}{\frak{p}})=1$ and $S_{\infty,\frak{p}} = \emptyset$. If $\mathcal{K}_{M+\frak{p}}$ is not compact then $\dim(\frac{A}{\frak{p}}) = 1$,
$S_{\infty,\frak{p}} \ne \emptyset$, and $f\in A$ is bounded on $\mathcal{K}_{M+\frak{p}}$ iff $v(f) \ge 0$ for all $v\in S_{\infty,\frak{p}}$.
(This uses  the compactness of the real spectrum.) Anyway, since $\mathcal{K}_M$ is the union of the $\mathcal{K}_{M+\frak{p}}$, we have established the following:

\smallskip
Claim 1: $f\in A$ is bounded on $\mathcal{K}_M$ iff $v(f)\ge 0$ holds for all $v \in S_{\infty}$.

\smallskip
If $S_{\infty} = \emptyset$ then every element of $A$ is bounded on $\mathcal{K}_M$, by Claim 1, so, by hypothesis, $A = \Bbb{R}\cdot 1$ (i.e., either
$A = M= \{ 0 \}$ or $A =\Bbb{R}$ and $M = \Bbb{R}^+$). In this case $M$ is obviously stable. Thus we may assume $S_{\infty} \ne \emptyset$. Let $\frak{p}_1,\dots,\frak{p}_k$ be the minimal primes of $A$ with $\dim(\frac{A}{\frak{p}_i}) =1$ and $S_{\infty,\frak{p}_i} \ne \emptyset$. 
Since $S_{\infty} \ne \emptyset$, $k\ge 1$. If $f \in \cap_{i=1}^k \frak{p}_i$ then $v(f) = \infty$ for all $v\in S_{\infty}$ so, by Claim 1, $f$ is
bounded on $\mathcal{K}_M$, so, by hypothesis, $f \in \Bbb{R}$. Since $k\ge 1$, this forces $f=0$. This proves $\cap_{i=1}^k \frak{p}_i = \{ 0 \}$, i.e.,
$\sqrt{\{ 0 \}} = \{ 0 \}$ and $\{ \frak{p}_i \mid i=1,\dots,k\}$ is the complete set of minimal primes of $A$.

For any non-empty subset $S$ of $S_{\infty}$ and any integer $d$,
let $$V_{S,d} := \{ f \in A \mid v(f)\ge d \ \forall v\in S\}.$$
$V_{S,d}$ is clearly an $\Bbb{R}$-subspace of $A$.

\smallskip

Claim 2: If $d<e$ then $V_{S,d}/V_{S,e}$ is finite dimensional.
\smallskip

Consider all pairs $(T,n)$ where $T$ is a non-empty subset of $S$
and $d\le n <e$ such that there exists an element $g \in A$ with
$v(g)=n$ for all $v\in T$ and $v(g)>n$ for $v\in S\backslash T$.
Fix such an element $g = g_{T,n}$ for each such pair. To prove
Claim 2 it suffices to show that these elements generate $V_{S,d}$
modulo $V_{S,e}$. This is pretty clear. Suppose $f\in V_{S,d}$.
Let $n = \min\{ v(f) \mid v\in S\}$, so $n\ge d$. If $n\ge e$ then
$f \in V_{S,e}$. Suppose $n<e$. Let $T = \{ v \in S \mid
v(f)=n\}$. Thus $T \ne \emptyset$. Fix $v_0 \in T$. Since the
residue field of $v_0$ is $\Bbb{R}$, there is some $a\in \Bbb{R}$
such that $v_0(f-ag_{T,n}) >n$. Let $f' = f-ag_{T,n}$, i.e., $f =
ag_{T,n}+f'$. Now repeat the process, working with $f'$ instead of
$f$. Either $\min \{ v(f') \mid v\in S \} >n$ or $\min\{ v(f')
\mid v\in S\} = n$ and $T' = \{ v \in S \mid v(f')=n\}$ is
non-empty and properly contained in $T$ (because $v_0 \in T$, $v_0
\notin T'$). Anyway, the process terminates after finitely many
steps. This proves Claim 2.

By Claim 1, $V_{S_{\infty},0} = \Bbb{R}$. Combining this with Claim 2, we see that $V_{S_{\infty},d}$ is finite dimensional for
each $d\le 0$. Clearly $A = \cup_{d\le 0} V_{S_{\infty},d}$.

Fix generators $g_1,\dots,g_t$ for $M$. We may assume each $g_i$
is $\ne 0$. Complications arise from the fact that $k$ may be
strictly greater than $1$, so some of the $g_i$ may be divisors of
zero. We need some notation: Let $g_0 := 1$. For $0\le i \le s$,
denote by $S_{\infty}^{(i)}$ the union of the sets
$S_{\infty,\frak{p}_j}$ such that $1\le j \le k$ and $g_i \notin
\frak{p}_j$. Thus $S_{\infty}^{(0)} = S_{\infty}$. Let $e_i :=
\max\{ v(g_i) \mid v\in S_{\infty}^{(i)}\}$.  Let $e_i' := \min\{
v(g_i) \mid v\in S_{\infty}^{(i)}\}$. Note that $S_{\infty}^{(i)}
\ne \emptyset$ so $e_i \ne +\infty$. Fix $d\ge 0$ and let
$W^{(i)}$ be a f.d. vector subspace of $A$ which generates
$V_{S_{\infty}^{(i)}, \frac{-d-e_i}{2}}$ modulo
$V_{S_{\infty}^{(i)}, -e_i'+1}$. This exists by Claim 2. To
complete the proof it suffices to prove:

\smallskip
Claim 3: Each $f \in M\cap V_{S_{\infty},-d}$ is expressible in the form $f=\sum_{i=0}^s \tau_ig_i$ where $\tau_i$ a sum of squares of elements
of $W^{(i)}$, $i=0,\dots,s$.

\smallskip
Let $f \in V_{S_{\infty},-d}$, $f = \sum_{i=0}^s \sigma_ig_i$,
$\sigma_i \in \sum A^2$. Then $-d \le v(f) = \min\{ v(\sigma_ig_i)
\mid i=0,\dots,s\}$, i.e., $v(\sigma_ig_i) \ge -d$ $\forall$ $v
\in S_{\infty}$. For $v\in S_{\infty}^{(i)}$ this yields
$v(\sigma_i) \ge -d-v(g_i) \ge -d-e_i$, by definition of $e_i$.
Express $\sigma_i$ as $\sigma_i = \sum h_{ip}^2$. Then
$v(h_{ip}^2) \ge -d-e_i$, i.e., $v(h_{ip}) \ge \frac{-d-e_i}{2}$
$\forall$ $v \in S_{\infty}^{(i)}$. Decompose $h_{ip}$ as $h_{ip}
= t_{ip}+ u_{ip}$ with $t_{ip} \in W^{(i)}$, $u_{ip} \in
V_{S_{\infty}^{(i)}, -e_i'+1}$. Then $h_{ip}^2 =
t_{ip}^2+2t_{ip}u_{ip}+u_{ip}^2= t_{ip}^2+(2t_{ip}+u_{ip})u_{ip}$,
so $h_{ip}^2g_i = t_{ip}^2g_i+(2t_{ip}+u_{ip})u_{ip}g_i$. One
checks that $v(u_{ip}g_i)>0$ for all $v\in S_{\infty}$. If
$v\notin S_{\infty}^{(i)}$ then $v(g_i)=\infty$, $v(u_{ip}g_i) =
\infty$, so this is clear. If $v\in S_{\infty}^{(i)}$, then
$v(u_{ip}) > -e_i'$, so $v(u_{ip}g_i) >-e_i'+v(g_i) \ge 0$ by
definition of $e_i'$.  According to Claim 1 and our hypothesis
this implies $u_{ip}g_i =0$. Thus $h_{ip}^2g_i = t_{ip}^2g_i$ and
$\sigma_ig_i = \tau_ig_i$ where $\tau_i := \sum t_{ip}^2$.
\end{proof}
 The
conclusion of Theorem \ref{4.1} is false if $\dim(\frac{A}{M\cap
-M})\geq 2$:

\begin{ex}\label{4.2} Let $A = \Bbb{R}[x,y]$, let $M$ be the preordering of $\Bbb{R}[x,y]$ generated by $(1-x)xy^2$, and let $N$ be the preordering of
$\Bbb{R}[x,y]$ generated by $(1-x)x$. $\mathcal{K}_N$ is the strip
$[0,1]\times \Bbb{R}$. $\mathcal{K}_M$ is the strip together with
the $x$-axis. Applying Schm\"udgen's fibre theorem (Theorem
\ref{new}) we see that $\overline{N} =
\operatorname{Pos}(\mathcal{K}_N)$. In fact, one even has $N =
\operatorname{Pos}(\mathcal{K}_N)$; see \cite{M1}. According to
\cite[Theorem 5.4]{Sc}, this implies that $N$ is not stable. On
the other hand, $y^2N \subseteq M$, so if $M$ were stable then $N$
would also be stable. (If $f\in N$ then $y^2f \in M$. If $M$ were
stable we would have $y^2f = \sigma+\tau (1-x)xy^2$, $\sigma, \tau
\in \sum \Bbb{R}[x,y]^2$ with degree bounds on $\sigma$ and $\tau$
depending only on $\deg(f)$. Clearly $\sigma =y^2\sigma_1$ for
some $\sigma_1\in \sum \Bbb{R}[x,y]^2$, so this would yield
$f=\sigma_1+\tau(1-x)x$ with degree bounds on $\sigma_1,\tau$
depending only on $\deg(f)$). This proves that $M$ is not stable.
On the other hand, the elements of $\Bbb{R}[x,y]$ bounded on
$\mathcal{K}_N$ are precisely the elements of $\Bbb{R}[x]$. Since
the only elements of $\Bbb{R}[x]$ bounded on the $x$-axis are the
elements of $\Bbb{R}$, this proves that the only elements of
$\Bbb{R}[x,y]$ bounded on $\mathcal{K}_M$ are the elements of
$\Bbb{R}$. We remark that even though $M$ is not stable, it might
still be closed.
\end{ex}

We can strengthen the example in the following way:

\begin{ex}\label{couex}  Let $A = \Bbb{R}[x,y]$, let $M$ be the preordering of $\Bbb{R}[x,y]$ generated by $(1-x)x^3y^2$ and let $N$ be the preordering of
$\Bbb{R}[x,y]$ generated by $(1-x)x^3$. Again $\mathcal{K}_N$ is
the strip $[0,1]\times \Bbb{R}$ and $\mathcal{K}_M$ is the strip
together with the positive part of the $x$-axis. Theorem \ref{new}
again shows  that $\overline{N} =
\operatorname{Pos}(\mathcal{K}_N)$. However, we have
$x\in\operatorname{Pos}(\mathcal{K}_N)\setminus N$. Indeed writing
down a possible representation of  $x$ in $N$ and evaluating in
$y=0$ gives such a representation for $x$ in $\Bbb{R}[x]$;
evaluating in $x=0$ then shows that $x^2$ divides $x$, a
contradiction. So $N$ can not be closed.

We have $N=\left\{f\in\Bbb{R}[x,y]\mid y^2f\in M\right\}$, with
the same argument as in the preceding example. Now since $N$ is
not closed and the mapping $f\mapsto y^2f$ is linear and therefore
continuous, $M$ can not be closed (so in view of Theorem
\ref{2.6}, $M$ can also not be stable). On the other hand the only
polynomials bounded on $\mathcal{K}_M$ (or $\mathcal{Y}_M$) are
the reals.
\end{ex}

Open problem 1 in \cite[p. 85]{Pl} should be mentioned in this context. It is asked  there whether the absence of nontrivial bounded polynomials implies stability of the quadratic module, at least if the semialgebraic set is regular at infinity. Our example does not answer the question, since $\mathcal{K}_M$ is not regular at infinity, i.e. it is not the union of a compact set and a set that is the closure of its interior. So the  question is still open.

For polyhedra however, the following result is true:

\begin{theorem}\label{poly} Let $A=\Bbb{R}[\underline{x}]$ and suppose $M$ is
generated by finitely many linear polynomials. Suppose the only
linear polynomials that are bounded on $\mathcal{K}_M$ are from $\Bbb{R} + M\cap -M$. Then $M$ is stable.
\end{theorem}
\begin{proof}If $\mathcal{K}_M$ has empty interior, then it lies in a strict affine subspace of $\Bbb{R}^n$.
Any linear polynomial vanishing on this subspace belongs to $M\cap -M$, by \cite[Lemma 7.1.5]{M}. So as explained before we
can assume that $\mathcal{K}_M$ has non-empty interior, and so
$M\cap -M=\{0\}.$

 Without loss of generality $0\in\mathcal{K}_M$. Group
the non constant linear generators of $M$ so that
$$p_1(0),\ldots,p_r(0)>0$$ and
$$q_1(0)=\cdots=q_s(0)=0.$$ Write $p_i=c_i +\widetilde{p}_i$ with $c_i\in\Bbb{R}_{>0}$ and
$\widetilde{p}_i(0)=0, \widetilde{p}_i\neq 0$. All
$\widetilde{p}_i$ and $q_j$ are homogeneous polynomials of degree
one. We claim that
$\widetilde{p}_1,\ldots,\widetilde{p}_r,q_1,\ldots,q_s$ are
positively linear independent. So assume $$\sum_i
\lambda_i\widetilde{p}_i +\sum_j \gamma_j q_j=0$$ for some
nonnegative coefficients $\lambda_i, \gamma_j$, not all zero. Then
some $\lambda_i$ must be nonzero, since $M\cap -M=\{0\}.$ Assume
$\lambda_1>0$. With $N:=\sum_i\lambda_ic_i$ we have
$\lambda_1p_1,N-\lambda_1p_1 \in M$. So by our assumption
$p_1\in\Bbb{R}$, a contradiction. This proves the claim.

So there must be a point $d\in\Bbb{R}^n$ where all
$\widetilde{p}_1,\ldots,\widetilde{p}_r,q_1,\ldots,q_s$ are
strictly positive (Theorem of alternatives for strict linear inequalities \cite[Example 2.21]{BV}). Thus $\mathcal{K}_M$
contains a full dimensional cone, and so $M$ is stable (see
\cite{KM} or \cite{N2} or \cite{PS}).
\end{proof}

We define the \it weak closure \rm $\widetilde{M}$ of a quadratic
module $M$ of $A$. Informally, $\widetilde{M}$ is the part of
$\overline{M}$ that
 can be `seen' by applying Theorem \ref{new} recursively. 
Formally, we define $\widetilde{M}$ as follows:
\begin{enumerate}
\item If $M=A$ then $\widetilde{M} = M$. \item If the only
elements of $A$ bounded on $\mathcal{Y}_M$ are the elements of
$\Bbb{R}+M\cap -M$, then $\widetilde{M} = M$. \item If some $f \in
A$ is bounded on $\mathcal{Y}_M$, say $a\le f\le b$ on
$\mathcal{Y}_M$, and $f \notin \Bbb{R}+M\cap -M$, then
$\widetilde{M} = \cap_{a\le \lambda \le b}
\widetilde{M_{\lambda}}$, where $M_{\lambda} := M+(f-\lambda)$.
\end{enumerate}

Note: Although case (1) is included for clarity, it can also be viewed as a special case of (2). It is also important to note that the description of $\widetilde{M}$ given in (3) holds trivially if $f \in \Bbb{R}+M\cap -M$ (in the sense that if $f=\lambda_0+g$, $\lambda_0 \in \Bbb{R}$, $g\in M\cap-M$, then $M_{\lambda} = M$ if $\lambda = \lambda_0$ and $M_{\lambda} = A$ if $\lambda \ne \lambda_0$).

\begin{thm}\label{4.3}Let $A$ be finitely generated. Then for every quadratic module $M$ of $A$,
\begin{enumerate}
\item $\widetilde{M}$ is a well-defined quadratic module of $A$.
\item $M \subseteq \widetilde{M} \subseteq \overline{M}$.
\end{enumerate}
\end{thm}

\begin{proof} $A$ is noetherian, so if the above notion of $\widetilde{M}$ not well defined, then there is some quadratic module $M$ with $M\cap -M$
 maximal such that $\widetilde{M}$ is not a well-defined quadratic module. Obviously we are not in case (1) or (2), i.e., we are in case (3). Suppose we have
  $f, g\in A$ bounded on $\mathcal{Y}_M$, say $a\le f \le b$ and $c\le g \le d$ on $\mathcal{Y}_M$, $f,g \notin \Bbb{R}+M\cap -M$. By the
  maximality of $M\cap -M$, $\widetilde{M+(f-\lambda)}$ and $\widetilde{M+(g-\mu)}$ are well-defined, $a\le \lambda \le b$, $c\le \mu \le d$. We need to show $$\cap_{a\le \lambda\le b}\widetilde{M+(f-\lambda)} = \cap_{c\le \mu \le d} \widetilde{M+(g-\mu)}.$$ This follows easily from $\widetilde{M+(f-\lambda)} = \cap_{c\le \mu \le d} \widetilde{M+(f-\lambda, g-\mu)}$ and $\widetilde{M+(g-\mu)} = \cap_{a\le \lambda \le b}\widetilde{M+(f-\lambda,g-\mu)}$. These latter facts hold either by definition or  for trivial reasons.

Statement (2) is proven similar. To prove $\widetilde{M} \subseteq
\overline{M}$ one of course needs to use Theorem \ref{new}.
\end{proof}

In \cite[Lemma 3.13]{Sc} it is shown that $\sqrt{M\cap-M}
\subseteq \overline{M}$, for arbitrary $A$ and $M$. One can
improve on this as follows:

\begin{lemma}\label{4.4} $\sqrt{M\cap-M} \subseteq \widetilde{M}$.
\end{lemma}

\begin{proof} Let $f \in \sqrt{M\cap-M}$. If $f=\lambda_0+g$, $\lambda_0 \in \Bbb{R}$, $g\in M\cap-M$, then $\lambda_0 = f-g \in \sqrt{M\cap-M}$. Then either $M\cap -M = A$ (so $f\in \widetilde{M}$) or $\lambda_0 = 0$ and $f\in M\cap-M \subseteq M \subseteq \widetilde{M}$. If $f \notin \Bbb{R}+M\cap-M$, then, since $0\le f \le 0$ on $\mathcal{Y}_M$, $\widetilde{M} = \cap_{0\le \lambda \le 0} \widetilde{M_{\lambda}} = \widetilde{M_0}$, where $M_{\lambda} := M+(f-\lambda)$. Anyway, since $f\in M_0$, this means $f\in \widetilde{M_0} = \widetilde{M}$.
\end{proof}

Note that example \ref{couex} above shows that $\widetilde{M}$ and
$\overline{M}$ are not the same in general. In the example the
only polynomials bounded on $\mathcal{Y}_M$ are the reals, so
$M=\widetilde{M}$. But we have shown that $M$ is not closed, so
$M=\widetilde{M}\subsetneq M^{\ddagger}\subseteq\overline{M}$
holds.

Note also that the inclusion $\widetilde{M}\subseteq M^{\ddagger}$
is not always true. Let $M$ be the preordering of $\Bbb{R}[x,y]$
generated by $y^3, x+y, 1-xy, 1-x^2$. This is the example from
\cite{N3} with $M^{\ddagger}\subsetneq\overline{M}$. One easily
checks that $\widetilde{M}=\overline{M}$ holds, so
$M^{\ddagger}\subsetneq\widetilde{M}$ in this example.

On the other hand, in many simple cases where we are able to
compute $\widetilde{M}$, we find $\widetilde{M}=\overline{M}$:

\begin{thm}\label{4.5} Let $A$ be finitely generated. $\overline{M} = \widetilde{M}$ holds in the following cases:
\begin{enumerate}
\item $M$ finitely generated and stable.
\item $M$ finitely generated and $\dim(\frac{A}{M\cap -M}) \le 1$.
\item  $M$ archimedean.
\item $A = \Bbb{R}[\underline{x}]$ and $M$ is generated by finitely many linear polynomials.
\end{enumerate}
\end{thm}

 \begin{proof} (1) $\overline{M} =M+\sqrt{M\cap-M}$, by Theorem \ref{2.6}. By Lemma \ref{4.4} this implies $\overline{M} \subseteq \widetilde{M}$.

(2) Choose $M$ with $M\cap -M$ maximal such that
$\dim(\frac{A}{M\cap -M}) \le 1$ and $\overline{M} \ne
\widetilde{M}$. In view of Ths. \ref{new} and \ref{4.1} and (1)
and the recursive description of $\widetilde{M}$, we again have
$\overline{M} = \widetilde{M}$, which is a contradiction.

(3) Choose $M$ with $M\cap -M$ maximal such that $M$ is
archimedean, $\overline{M} \ne \widetilde{M}$. Since $M$ is
archimedean every element of $A$ is bounded on $\mathcal{Y}_M$. If
$A = \Bbb{R} +M\cap -M$ then $\overline{M} = M= \widetilde{M}$, a
contradiction. If there is some $f\in A$, $f\notin \Bbb{R}
+M\cap-M$, then by Theorem \ref{new} and the recursive description
of $\widetilde{M}$, $\overline{M} = \widetilde{M}$, again a
contradiction.

(4) Take such $M$ with $M\cap-M$ maximal such that
$\widetilde{M}\neq\overline{M}$. So again the only elements
bounded on $\mathcal{Y}_M$ are the elements from $\Bbb{R}+M\cap
-M$. Any linear polynomial that is bounded on $\mathcal{K}_M$ is
also bounded on $\mathcal{Y}_M$. So by Theorem \ref{poly}, $M$ is
stable, and we are in case (1).
\end{proof}

\section{Appendix}

In this section we construct a cone for which  the sequence
of iterated sequential closures terminated after $n$ steps.
Therefore let
$$E=\bigoplus_{i=0}^{\infty} \R\cdot e_i =\left\{
(f_i)_{i\in\N}\mid f_i\in\R, \mbox{ only finitely many } f_i\neq
0\right\}$$ be a countable dimensional $\R$-vector space. For
$m\in \N\setminus\{0\}$ we write
$$W_m:=\bigoplus_{i=0}^{m-1}\R\cdot e_i,$$ so the increasing
sequence $(W_m)_{m\in \N}$ of finite dimensional subspaces
exhausts the whole space $E$.
 For $n\in\{1,2,\ldots\}$ and
$l=(l_0,l_1,\ldots,l_n)\in\left(\N\setminus\{0\}\right)^{n+1}$
define

\begin{align*}V(l)&:= \underbrace{[\frac{1}{l_1},1]\times \cdots\times
[\frac{1}{l_1},1]}_{l_0 \mbox{ times }}\times
\underbrace{[\frac{1}{l_2},1]\times \cdots\times
[\frac{1}{l_2},1]}_{l_1 \mbox{ times }}\times\cdots \\ & \times
\underbrace{[\frac{1}{l_{n}},1]\times \cdots \times
 [\frac{1}{l_{n}},1]}_{l_{n-1} \mbox{ times
}}.\end{align*} $V(l)$ is a compact subset of $W_{l_0+\cdots
+l_{n-1}}$. Let
$$U(l):= V(l)\times\bigoplus_{i=l_0+\cdots+l_{n-1}}^{\infty}
[0,1]\cdot e_i, $$ so $U(l)\subseteq E$ and $U(l)\cap W_m$ is
compact for every $m\in\N$; indeed non-empty if and only if $m\geq
l_0+\cdots+l_{n-1}$.  Now define
 $$M_n:=\bigcup_{l\in \left(\N\setminus\{0\}\right)^{n+1}}
U(l).$$  The intention behind this is that $M_n$ contains $n$
"steps", and taking the sequential closure removes one at a time.

We have for $m\geq n\geq 2$ $$\overline{M_n\cap W_m}\subseteq
M_{n-1}.$$ To see this take a converging sequence $(x_i)_i$ from
$M_n\cap W_m$. So for each $x_i$ there is some $l^{(i)}\in
\left(\N\setminus\{0\}\right)^{n+1}$ such that $x_i\in
U(l^{(i)})$. As  $ U(l)\cap W_m$ is only non-empty if
$l_0+\cdots+l_{n-1}\leq m$, we can assume without loss of
generality (by choosing a subsequence), that the $l^{(i)}$
coincide in all but the last component. This shows that the limit
of the sequence $(x_i)_i$ belongs to $M_{n-1}$ (indeed to
$U(l_0^{(i)},\ldots, l_{n-1}^{(i)})\cap W_m$).

So $(M_n)^{\ddagger}\subseteq M_{n-1}$, and the other inclusion is
obvious. We thus have for $n\geq 2:$
$$(M_n)^{\ddagger}=M_{n-1}.$$ In addition, $$M_1\subsetneq M_1^{\ddagger}=
\bigoplus_{i=0}^{\infty} \ [0,1]\cdot e_i,$$ which is closed. This
shows that the sequence of sequential closures for $M_n$
terminates precisely after $n$ steps at  $\overline{M_n}=
\bigoplus_{i=0}^{\infty} \ [0,1]\cdot e_i$.

 Let $\cc(M_n)$ denote the
cone generated by $M_n$, i.e. $\cc(M_n)$ consists of all
finite positive combinations of elements from $M_n$, including
$0$. We have for $n\geq 2$
$$\cc(M_n)^{\ddagger}=\cc(M_{n-1}).$$
To see "$\subseteq$" suppose $x\in\cc(M_n)^{\ddagger}.$ Then we
have a sequence $(x_i)_i$ in some $\cc(M_n)\cap W_m=\cc(M_n\cap
W_m)$ that converges to $x$ in $W_m$. Write $$x_i =
\lambda_1^{(i)}a_1^{(i)} + \cdots \lambda_N^{(i)}a_N^{(i)}$$ with
all $a_j^{(i)}\in M_n\cap W_m$ and all $\lambda_j^{(i)}\geq 0$. We
can choose the same sum length $N$ for all $x_i$, by the conic
version of Carath\'{e}odory's Theorem (see for example \cite{Ba},
Problem 6, p. 65). By choosing a subsequence of $(x_i)_i$ we can
assume that for all $j\in\{1,\ldots,N\}$ the sequence
$(a_j^{(i)})_i$ converges to some element $a_j$. This uses
$M_n\cap W_m\subseteq [0,1]^m.$ All elements $a_j$ lie in
$M_n^{\ddagger}=M_{n-1}.$ As $n\geq 2$, the first component of
each element $a_j^{(i)}$ is at least $\frac1m$. So all the
sequences $(\lambda_j^{(i)})_i$ are bounded and therefore without
loss of generality also convergent. This shows that $x$ belongs to
$\cc(M_{n-1})$.

To see "$\supseteq$" note that $M_n^{\ddagger}\subseteq
\cc(M_n)^{\ddagger}$ and  $\cc(M_n)^{\ddagger}$ is a cone.
So
$$\cc(M_{n-1})=\cc(M_n^{\ddagger})\subseteq \cc(M_n)^{\ddagger}.$$ For $n=1$
we have
$$\cc(M_1)=\left\{f=(f_i)_i\in\bigoplus_{i=0}^{\infty}\R_{\geq 0}\cdot e_i
\mid f_0=0\Rightarrow f=0\right\},$$ so $$\cc(M_1)\subsetneq
\cc(M_1)^{\ddagger}= \bigoplus_{i=0}^{\infty}\R_{\geq 0}\cdot
e_i,$$ which is closed. So all in all we have proved:

\noindent  For $n\in\{1,2,\ldots\}$ and the cone
$\cc(M_n)$, the
 sequence of iterated sequential closures terminates precisely
 after $n$ steps at
$$\overline{\cc(M_n)} = \bigoplus_{i=0}^{\infty}\R_{\geq 0}\cdot
e_i.$$

\bn
\bn
\adresse

\end{document}